\documentclass{amsart}

\usepackage{amsmath, amssymb, amsthm}

\newcommand{\bP}{\mathbb{P}}
\newcommand{\bN}{\mathbb{N}}
\newcommand{\bC}{\mathbb{C}}
\newcommand{\bR}{\mathbb{R}}

\newcommand{\rd}{\mathrm{d}}

\newcommand{\supp}{\operatorname{supp}}
\newcommand{\Res}{\operatorname{Res}}

\theoremstyle{plain}

\newtheorem{lemma}[equation]{Lemma}

\newtheorem{claim}{Claim}
\newtheorem{mainth}{Theorem}
\theoremstyle{definition}

\newtheorem*{ac}{Acknowledgment}
\theoremstyle{remark}
\newtheorem{remark}[equation]{Remark}
\newtheorem{notation}[equation]{Notation}

\numberwithin{equation}{section}

\renewcommand{\bold}[1]{\smallskip \noindent {\bf \boldmath #1 }\nopagebreak[4]}

\begin{document}
\title[a characterization of polynomials]{Potential theory and a
characterization of polynomials in complex dynamics
}
\author[Y. Okuyama]{Y\^usuke Okuyama}
\address{
Department of Comprehensive Sciences,
Graduate School of Science and Technology,
Kyoto Institute of Technology,
Kyoto 606-8585 JAPAN}
\email{okuyama@kit.ac.jp}
\author[M. Stawiska]{Ma{\l}gorzata Stawiska}
\address{
Department of Mathematics, University of Kansas, 1460 Jayhawk Blvd., Lawrence, KS 66045, USA}
\email{stawiska@ku.edu}

\thanks{The first author is partially supported by JSPS Grant-in-Aid for
Young Scientists (B), 21740096.}

\date{\today}


\begin{abstract}
 We obtain a measure theoretical characterization of polynomials
 among rational functions on $\bP^1$,
 which generalizes a theorem of Lopes.
 Our proof applies both classical and
 dynamically weighted potential theory.
\end{abstract}

\maketitle

\section{Introduction}\label{sec:intro}

We are interested in
a measure theoretical characterization of polynomials
among rational functions $f$ of degree $d>1$ on $\bP^1$.
Recall that the Fatou set $F(f)$ of $f$ is the region of normality of
iterates $\{f^k;k\in\bN\}$ in $\bP^1$, which by the definition is open.
The Julia set $J(f)$ is the complement of $F(f)$ and it is
known to be non-empty. Both $F(f)$ and $J(f)$ are $f$-invariant.
The characterization that we have in mind is provided
by the following theorem:

\begin{mainth}\label{th:polynomial}
 Let $f$ be a rational function on $\bP^1$ of degree $d>1$.
 Suppose that the point $\infty$ belongs to a Fatou component
 $D_{\infty}$ of $f$,
 and that $f(D_{\infty})=D_{\infty}$.
 Then the following are equivalent$:$
\begin{enumerate}
 \item $f$ is a polynomial.
       \label{item:poly}
 \item The balanced measure $\mu_f$ of $f$  coincides with
       the harmonic measure $\nu$ of $D_\infty$ with pole $\infty$.
       \label{item:equal}
\end{enumerate}
\end{mainth}

The probability measure $\mu_f$ in Theorem \ref{th:polynomial} is
also known to be the (unique) maximal entropy measure of $f$, which was constructed 
by Lyubich \cite{Lyubich83} and by 
Freire, Lopes and Ma\~n\'e \cite{FLM83} (see the next paragraph
for more historical remarks).
Under the additional assumption that $f(\infty)=\infty$,
Theorem \ref{th:polynomial} was proved by Lopes \cite{Lopes86}. 
The Lopes's theorem was stated earlier in 
Oba and Pitcher \cite[Theorem 6]{ObaPitcher}, but proved only partially.
Lalley gave a probabilistic proof of Lopes's theorem 
(\cite[\S 6]{Lalley92}).
In all those proofs, a key role is played by the same equality,
which is a consequence of a pullback formula
for the logarithmic potential of balanced measure.
We will give a simple and conceptually new proof of
both this formula and the key equality in an improved form
(Lemma \ref{th:keygeneral} and Claim \ref{th:lemniscate}),
which will enable us to prove Theorem \ref{th:polynomial}.
Ma\~n\'e and da Rocha \cite{ManeDaRocha92} also
studied Lopes's theorem in relation to calculations
of the entropy of invariant measures on the Julia set. 

Brolin's theorem \cite[Theorem 16.1]{Brolin} says that
 the pullbacks $(f^k)^*\delta_a/d^k$
 converge weakly to the harmonic measure $\nu=\nu_{\infty}$ of
 $D_{\infty}$ with pole at $\infty$ 
 when $f$ is a polynomial and $\delta_a$ is the Dirac measure
 at a non-exceptional point $a\in\bC$ of $f$.
 As a generalization of Brolin's theorem, the balanced measure $\mu_f$ 
 was first obtained as the weak limit of pullbacks $(f^k)^*\delta_a/d^k$
 for a rational function $f$ and a non-exceptional point $a\in\bP^1$
 (by Lyubich \cite{Lyubich83} and by Freire, Lopes and Ma\~n\'e
 \cite{FLM83} independently).
 However, in this article, we will not use these equidistribution
 results.

In the next section we recall the definition of measures $\mu_f$ and $\nu$.
Our proof of Theorem \ref{th:polynomial} is
an application of both classical and dynamically weighted potential theory.
For related results on (generalized) polynomial-like maps,
see \cite{PopoviciVolberg96}, \cite{Zdunik97}.

\begin{notation}
 We denote the origin of $\bC^2$ by $0$.
 Let $\pi:\bC^2\setminus\{0\}\to\bP^1$ be the canonical projection
 so that $\pi(z_0,z_1)=z_1/z_0$ if $z_0\neq 0$ and
 $\pi(z_0,z_1)=\infty$ if $z_0=0$.
 Let $\|\cdot\|$ be the Euclidean norm on $\bC^2$, and
 put $(z_0,z_1)\wedge(w_0,w_1):=z_0w_1-z_1w_0$ on $\bC^2\times\bC^2$.
 A function on $\bP^1$ is said to be $\delta$-subharmonic (DSH)
 if it is locally the difference of two subharmonic functions.
 We normalize $\rd^c$ so that
 $\rd\rd^c=(i/\pi)\partial\overline{\partial}$.  An important example of a value of the $dd^c$-operator is
 the generalized Laplacian of the $\delta$-subharmonic function
 $\log|\cdot-w|$ $(w\in\bC)$ on $\bP^1$,  which equals  $\rd\rd^c\log|\cdot-w|=\delta_w-\delta_{\infty}$,
 where $\delta_w$ denotes the Dirac measure at $w\in\bP^1$.
\end{notation}

\section{Rational functions and probability measures on $\bP^1$}
\label{sec:facts}

\bold{Balanced measure $\mu_f$.}
For more details, see \cite[\S 4]{HP94}, \cite[\S 1]{Ueda94}
and \cite[Chapitre VIII]{BM01}.

Let $f$ be a rational function on $\bP^1$ of degree $d>1$.
A lift 
\begin{gather*}
 F(z_0,z_1)=(F_0(z_0,z_1),F_1(z_0,z_1))
\end{gather*}
of $f$ is a non-degenerate homogeneous polynomial endomorphism
of algebraic degree $d$ on $\bC^2$ in that
$\pi\circ F=f\circ\pi$ and $F^{-1}(0)=\{0\}$,
and is uniquely determined up to multiplication by
a constant in $\bC^*$.

The dynamical Green function of $F$ is
\begin{gather*}
 G^F:=\lim_{k\to\infty}\frac{1}{d^k}\log\|F^k\|:\bC^2\to\bR\cup\{-\infty\}.
\end{gather*}
This convergence is uniform on $\bC^2\setminus\{0\}$, so
$G^F$ is continuous there and plurisubharmonic on $\bC^2$.
It follows from the definition of $G^F$ that
\begin{gather}
 d\cdot G^F=G^F\circ F,\label{eq:invariance}
\end{gather}
and from homogeneity of $F$ that for every $p\in\bC^2$ and every $c\in\bC^*$,
\begin{gather}
\label{eq:scaled} G^F(c\cdot p)=G^F(p)+\log|c|,\\
 G^{cF}(p)=G^F(p)+\frac{1}{d-1}\log|c|.\label{eq:choice}
\end{gather}
The function $G^F(1,\cdot)$ is continuous on $\bC$ and
$\delta$-subharmonic on $\bP^1$, and the balanced measure $\mu_f$ is 
 defined by the unique probability measure on $\bP^1$ satisfying
\begin{gather} \label{eq:dd^c}
 \rd\rd^c G^F(1,\cdot)=\mu_f-\delta_{\infty}.
\end{gather}
It follows from (\ref{eq:choice}) that  the left-hand side of (\ref{eq:dd^c}) is independent of the choice of $F$, hence 
the measure $\mu_f$ is  well-defined.
It is also known that $\supp\mu_f=J(f)$
and $J(f)$ is perfect.

From (\ref{eq:invariance}),
the $\mu_f$ is balanced and invariant under $f$: namely,
\begin{gather}
 \frac{f^*\mu_f}{d}=\mu_f=f_*\mu_f.\label{eq:balanced}
\end{gather}
We recall that the pullback $f^*\phi$ of continuous function $\phi$ on $\bP^1$
is defined by $\phi\circ f$, and the push-forward $f_*\phi$ is 
\begin{gather*}
 f_*\phi(z):=\frac{1}{d^k}\sum_{w\in f^{-1}(z)}\phi(w),\label{eq:pullback}
\end{gather*}
where the sum takes into account of the multiplicity of $f$ at each $w$.
Both $f^*\phi$ and $f_*\phi$ are continuous on $\bP^1$,
and we may define respectively the push-forward $f_*\mu$
and the pullback $f^*\mu$ of (Radon) measure $\mu$ by duality.

\bold{Harmonic measure $\nu=\nu_{\infty}$.}
For a finite Borel measure $\nu$ on $\bC$
with compact support, its logarithmic potential on $\bC$ is
\begin{gather}
 p_{\nu}(z):=\int_{\bC}\log|z-w|\rd\nu(w)=\nu(\bC)\log|z|+O(|z|^{-1})\label{eq:logarithmic}
\end{gather}
as $z\to\infty$,
and the logarithmic energy of $\nu$ is $I_{\nu}:=\int_{\bC}p_{\nu}\rd\nu$.
A compact set $K$ in $\bC$ is said to be polar if
\begin{gather*}
\sup\left\{I_{\nu};\supp\nu\subset K,\nu(\bC)=1\right\}=-\infty.\label{eq:supremum}
\end{gather*}
If $K$ is non-polar, then by Frostman's theorem,
there is the unique probability measure
$\nu=\nu_K$ (the equilibrium measure of $K$)
which attains the supremum in the above.
The measure $\nu$ has the support
on the exterior boundary $\partial_e K$ of $K$, and
satisfies that $p_{\nu}\equiv I_{\nu}$ on $K\setminus E$, where $E$ is
a (possibly empty) $F_{\sigma}$ polar subset of $\partial_e K$. 
Moreover, $p_{\nu}>I_{\nu}$ on $D_{\infty}$ by the minimum principle.

For a domain $D$ in $\bP^1$ which contains $\infty$ and
 whose complement $\bC\setminus D$ is non-polar, 
the harmonic measure
$\nu=\nu_{\infty}$ of $D$ with pole $\infty$ is determined by
$\nu_{\bC\setminus D}$ (cf. \cite[Theorem 4.3.14]{Rans95}).
 Under the situation in Theorem \ref{th:polynomial},
we will compute $I_{\mu_f}$ and see directly that $I_{\mu_f}>-\infty$, 
so $\bC\setminus D_{\infty}$ is non-polar
(see Lemma \ref{th:logenergy} below).
Hence $D_{\infty}$ admits the harmonic measure 
$\nu=\nu_{\infty}(=\nu_{\bC\setminus D_{\infty}})$
with pole  at $\infty$. 

\bold{Dynamically weighted potential theory.}
A function
\begin{gather*}
 \bC^2\times\bC^2\ni(p,q)\mapsto\log|p\wedge q|-G^F(p)-G^F(q)\in\bR\cup\{-\infty\}
\end{gather*}
descends to a weighted kernel $\Phi_F(z,w)$
($p\in\pi^{-1}(z),q\in\pi^{-1}(w)$) on $\bP^1$.
For a Radon measure $\mu$, its $F$-potential is
a $\delta$-subharmonic function
\begin{gather*}
 U_{F,\mu}(z):=\int_{\bP^1}\Phi_F(z,w)\rd\mu(w).
\end{gather*}
It can be computed directly that
$\rd\rd^cU_{F,\mu}=\mu-\mu(\bP^1)\mu_f$,
so the potential $U_{F,\mu_f}$ of $\mu=\mu_f$ is harmonic
on $\bP^1$, and hence constant, say,
$U_{F,\mu_f}\equiv V_F$. The constant $V_F$ has been computed as
\begin{gather}
 V_F=-\frac{1}{d(d-1)}\log|\Res F|\label{eq:resultant}
\end{gather}
in \cite[Theorem 1.5]{DeMarco03}. We will compute it in a different way in
Appendix.

\section{A proof of Theorem \ref{th:polynomial}}\label{sec:proof}

Let $f$ be a rational function on $\bP^1$ of degree $d>1$, and
$F=(F_0,F_1)$ a lift of $f$. We note that
\begin{gather*}
 f(z)=F_1(1,z)/F_0(1,z).
\end{gather*}
Put $d_0:=\deg F_0(1,z)$ and $d_1:=\deg F_1(1,z)$,
and let $a_F,b_F$ be the coefficients
of the maximal degree term of $F_0(1,z),F_1(1,z)$, respectively.

Suppose that a Fatou component $D_{\infty}$ of $f$
contains $\infty$, and that $f(D_{\infty})=D_{\infty}$. 

\begin{lemma}\label{th:decomp}
 For every $z\in\bC$,
 \begin{gather*}
  p_{\mu_f}(z)=G^F(1,z)-G^F(0,1).\label{eq:Riesz}
 \end{gather*}
 In particular, $p_{\mu_f}$ is continuous on $\bC$.
\end{lemma}

\begin{proof}
 Recall that $U_{F,\mu_f}\equiv V_F$ on $\bP^1$.
 Hence for every $z\in\bC$,
 \begin{align*}
 p_{\mu_f}(z)&=\int_{\bC}\log|z-w|\rd\mu_f(w)\\
 &=U_{F,\mu_f}(z)+G^F(1,z)+\int_{\bC}G^F(1,w)\rd\mu_f(w)\\
 &=G^F(1,z)+C_F,
 \end{align*}
 where we put $C_F:=V_F+\int_{\bC}G^F(1,w)\rd\mu_f(w)$. Hence
 from (\ref{eq:scaled}) and (\ref{eq:logarithmic}),
 \begin{gather*}
  0=\lim_{z\to\infty}(p_{\mu_f}(z)-\log|z|)=\lim_{z\to\infty}G^F(1/z,1)+C_F,
 \end{gather*}
 so that $C_F=-G^F(0,1)$.
\end{proof}

The following computation of $I_{\mu_f}$ 
may be of independent interest, and
was proved in \cite[Theorem 4]{ObaPitcher}
under the restrictions $f(\infty)=\infty$ and $d_0<d-1$, with no reference to $G^F$.

\begin{lemma}\label{th:logenergy}
 The complement $\bC\setminus D_{\infty}$ of $D_{\infty}$is non-polar,
and $D_{\infty}$ admits the harmonic measure $\nu=\nu_{\infty}$
with pole $\infty$. The energy $I_{\mu_f}$ of $\mu_f$ is computed as
\begin{gather*}
e^{I_{\mu_f}}=e^{-2G^F(0,1)}|\Res F|^{\frac{1}{d(d-1)}}.\label{eq:energy}
\end{gather*}
Here $\Res F:=a_F^{d-d_1}b_F^{d-d_0}R(F_0(1,z),F_1(1,z))$ is the
homogeneous resultant of $F$,
where $R(P(z),Q(z))$ is the resultant of two polynomials $P(z)$ and $Q(z)$.
\end{lemma}

\begin{proof}
Integrating the equality in Lemma \ref{th:decomp} in $\rd\mu_f(z)$, we get
\begin{gather*}
I_{\mu_f}=\int_{\bC} p_{\mu_f}\rd\mu_f=
\int_{\bC}G^{F}(1,\cdot)\rd\mu_f-G^F(0,1)(>-\infty),
\end{gather*}
 which with $\supp\mu_f\subset J(f)\subset\bC\setminus D_{\infty}$
implies that $\bC\setminus D_{\infty}$ is non-polar.
It also follows from the proof of the previous lemma that
\begin{gather*}
-G^F(0,1)=C_F=V_F+\int_{\bC}G^F(1,\cdot)\rd\mu_f,
\end{gather*}
where $C_F$ has been introduced in the proof of the previous lemma. 
The value of $V_F$ has already been computed in (\ref{eq:resultant}). Hence
\begin{gather*}
I_{\mu_f}=\frac{1}{d(d-1)}\log|\Res F|-2G^F(0,1).
\end{gather*}
\end{proof}

The following was 
proved in \cite[p307]{ObaPitcher} 
and \cite[p398]{Lopes86} under the assumption $f(\infty)=\infty$ and
$z\in J(f)$ (again with no reference to $G^F$).

\begin{lemma}\label{th:keygeneral}
For every $z\in\bC\setminus f^{-1}(\infty)$,
\begin{gather*}
 p_{\mu_f}(f(z))=d\cdot p_{\mu_f}(z)-\log|F_0(1,z)|+(d-1)G^F(0,1).
 \label{eq:direct}
\end{gather*}
\end{lemma}

\begin{proof}
 For every $z\in\bC\setminus f^{-1}(\infty)$, from (\ref{eq:invariance}),
\begin{align*}
   G^{F}(1,f(z))
  =G^{F}(F(1,z))-\log|F_0(1,z)|
  =d\cdot G^{F}(1,z)-\log|F_0(1,z)|.
\end{align*}
 Now Lemma \ref{th:decomp} completes the proof.
\end{proof}

We now proceed with the proof of our theorem.
Let us denote the harmonic measure of $D_{\infty}$ with pole at $\infty$
by $\nu=\nu_{\infty}$.

\bold{(\ref{item:equal})$\Rightarrow$(\ref{item:poly}).}
Suppose first that $\mu_f=\nu$.
Under this assumption, $p_{\nu}=p_{\mu_f}$ is continuous
on $\bC$ (Lemma \ref{eq:Riesz}), so we have 
$p_{\mu_f}=p_{\nu}\equiv I_{\nu}=I_{\mu_f}$ on $\bC\setminus D_{\infty}$.

A weaker form of the identity in Claim \ref{th:lemniscate}
was the key equality which we mentioned in \S \ref{sec:intro}
(\cite[p398]{Lopes86}, \cite[Claim 2 in \S 6]{Lalley92}).

\begin{claim}\label{th:lemniscate}
On $\bC\setminus f^{-1}(D_{\infty})$,
\begin{gather*}
|F_0(1,\cdot)|\equiv e^{(d-1)(I_{\mu_f}+G^F(0,1))}.\label{eq:lemniscate}
\end{gather*}
\end{claim}

\begin{proof}
 The assumption $f(D_{\infty})=D_{\infty}$
 implies $D_{\infty}\subset f^{-1}(D_{\infty})$. Hence
 $p_{\mu_f}\circ f=p_{\mu_f}\equiv I_{\mu_f}$ on 
 $\bC\setminus f^{-1}(D_{\infty})$.
 Now Lemma \ref{th:keygeneral} completes the proof.
\end{proof}

Suppose that $F_0(1,\cdot)$ is non-constant, that is, $d_0>0$.
Our refinement of the key equality make the following reduction possible.

\begin{claim}\label{th:reduction}
 $F(f)=D_{\infty}$.
\end{claim}

\begin{proof}
 If $F(f)\neq f^{-1}(D_{\infty})$, then
 $F(f)\setminus f^{-1}(D_{\infty})$
 is a non-empty open subset of $\bC\setminus f^{-1}(D_{\infty})$.
 By the identity theorem, Claim \ref{th:lemniscate} implies that $F_0(1,\cdot)$ must be constant,
 which contradicts the assumption $d_0>0$. Hence
 $F(f)=f^{-1}(D_{\infty})$, which implies that
 $F(f)=f(F(f))=D_{\infty}$.
\end{proof}

By Claim \ref{th:lemniscate},
$J(f)$ is contained in the lemniscate
\begin{gather*}
 L:=\{z\in\bC;|F_0(1,z)|=e^{(d-1)(I_{\mu_f}+G^F(0,1))}\}.
\end{gather*}
We note that each component of $L$ is a (possibly non-simple)
closed curve in $\bC\setminus f^{-1}(\infty)$,  which is
 real-analytic except  for finitely many singularities;
more precisely, 
for each $z_0\in L$, there are an $n\in\bN$, a M\"obius transformation $K$
and a local holomorphic coordinate $h$ around $z_0$  
such that $h(z_0)=0$, 
$K(F_0(1,z_0))=0$, $\mathrm{Im}K(F_0(1,\cdot))\equiv 0$ on $L$ and
$K(F_0(1,h^{-1}(w)))=w^n$ around $0$. In particular,
$L$ around $z_0$ 
is the image of $\bigcup_{-n\le j<n}\{w;\arg w=j\pi/n\}\cup\{0\}$ under
$h^{-1}$, and $n\ge 2$ if and only if $z_0$ is a critical point of
$F_0(1,\cdot)$.

Fix a component $l$ of $L$ intersecting $J(f)$.

\begin{claim}\label{th:proper}
 For every $k\in\bN$, $f^k(l)\subset L$.
\end{claim}

\begin{proof}
 Fix $z_0\in l\cap J(f)$. 
 Since $L$ has at most $d_0$ components,
 there exists $\delta>0$ such that
 $\{|z-z_0|<\delta\}\cap J(f)\subset l$. Hence 
 from the perfectness of $J(f)$,
 $z_0$ is a non-isolated point of $l\cap J(f)$.
 
 For every $k\in\bN$, let $L_k$ be the component of $L$ containing $f^k(z_0)$.
 By the same argument as the above, there is  $\delta_k>0$ such that
 $\{|z-f^k(z_0)|<\delta_k\}\cap J(f)\subset L_k$. Hence
 if $\delta>0$ is small enough, then
 $f^k(\{|z-z_0|<\delta\}\cap l\cap J(f))\subset L_k$.
 Then by an argument involving the identity theorem \cite{Minda77},
 $f^k(l)\subset L_k(\subset L)$.
\end{proof}

\begin{claim}
 $l\subset J(f)$.
\end{claim}
\begin{proof}
  Suppose that $l\cap F(f)\neq\emptyset$. Then
 there exists $z_0\in l\cap D_{\infty}$ by Claim \ref{th:reduction}, 
 and hence $p_{\mu_f}(z_0)=p_{\nu}(z_0)>I_{\nu}=I_{\mu_f}$.
 By Claim \ref{th:proper}, for every $k\in\bN$,
 we have $f^{k-1}(z_0)\subset L\subset\bC\setminus f^{-1}(\infty)$,
 which with Lemma \ref{th:keygeneral}
 (and the definition of $L$) implies that
 $p_{\mu_f}(f^k(z_0))-I_{\mu_f}=d\cdot(p_{\mu_f}(f^{k-1}(z_0))-I_{\mu_f})$.
 Hence
 \begin{gather*}
  p_{\mu_f}(f^k(z_0))-I_{\mu_f}=d^k\cdot(p_{\mu_f}(z_0)-I_{\mu_f})>0,
 \end{gather*} 
 so $\lim_{k\to\infty}p_{\mu_f}(f^k(z_0))=\infty$.
 On the other hand, by Claim \ref{th:proper},
 we have $(f^k(z_0))\subset L$, and 
 since $p_{\mu_f}$ is upper semicontinuous,
 we have $\sup_{k\in\bN}p_{\mu_f}(f^k(z_0))\le\sup_L p_{\mu_f}<\infty$.
 This is a contradiction.
\end{proof}

Let $U$ be a component of $\bP^1\setminus l$ not containing $\infty$.
Then $\partial U\subset J(f)$. By the maximum modulus principle,
it  follows that $U\subset\bC\setminus L$, and from $J(f)\subset L$,
we have $U\subset F(f)$. Hence $U$ is a Fatou component of $f$, 
and by Claim \ref{th:reduction}, we must have $U=F(f)=
D_{\infty}$. This contradicts our assumption $\infty\not\in U$. 

Now the proof of (\ref{item:equal})$\Rightarrow$(\ref{item:poly}) is complete.

\bold{(\ref{item:poly})$\Rightarrow$(\ref{item:equal}).}
 Here we will give a proof of the assertion $\mu_f=\nu$ 
 without the equidistribution results mentioned in
 Section \ref{sec:intro}, using only computation
 of capacity.
 
Suppose that $f$ is a polynomial, or equivalently,
that $F_0(1,z)\equiv a_F$ on $\bC$. By a direct computation,
Lemma \ref{th:logenergy} implies that
\begin{gather*}
 e^{I_{\mu_f}}=\exp\left(-2\cdot\frac{1}{d-1}\log|b_F|\right)\cdot(|a_Fb_F|^d)^{\frac{1}{d(d-1)}}
=|a_F/b_F|^{\frac{1}{d-1}}.
\end{gather*}
Brolin's theorem 
which we mentioned in Section \ref{sec:intro} was based on the computation of $e^{I_{\nu}}$
(\cite[Lemma 15.1]{Brolin}) as
\begin{gather*}
 e^{I_{\nu}}=|b_F/a_F|^{-\frac{1}{d-1}}.
\end{gather*}
Hence $I_{\mu_f}=I_{\nu}$, and from the uniqueness of
$\nu=\nu_{\bC\setminus D_\infty}$, we have $\mu_f=\nu$.

Now the proof of Theorem \ref{th:polynomial} is complete. \qed

\begin{ac}
 We thank David Drasin for his comments on an earlier version of this paper, which helped us improve the presentation of our results.
\end{ac}

\section{Appendix: a computation of $V_F$}
Let $f$ be a rational function of degree $d>1$, and $F=(F_0,F_1)$ be a lift of $f$. 
For completeness, we give a direct computation (\ref{eq:resultant}) of $V_F$, 
again without using the equidistribution theorem. Indeed,
we can give a proof of the equidistribution theorem based on (\ref{eq:resultant}).
For the original computation of $V_F$, which uses the equidistribution theorem,
see DeMarco \cite[Theorem 1.5]{DeMarco03}.
After having written this appendix,
we learned that similar computations and formulas appeared in Appendix A
in \cite{Baker09}.

We continue to use the notation
$d_0,d_1,a_F,b_F$ as in \S \ref{sec:proof}.
Let us write as $F_0(1,z)=a_F\prod_{j=1}^{d_0}(z-w_j)$. Then
since $F:\bC^2\to\bC^2$ is homogeneous and
of topological degree $d^2$, we get on $\bC^2$,
\begin{align*}
 |F(p)\wedge (0,1)|
=&|F_0(p)|
=|a_F||p\wedge (0,1)|^{d-d_0}\prod_j|p\wedge (1,w_j)|\\
=&(|a_F|(|F_1(0,1)|^{1/d})^{d-d_0}\prod_j|F_1(1,w_j)|^{1/d})
\prod_{q\in F^{-1}(0,1)}|p\wedge q|^{1/d},
\end{align*}
and the leading coefficient is computed as
\begin{gather*}
 |a_F|(|F_1(0,1)|^{1/d})^{d-d_0}\prod_j|F_1(1,w_j)|^{1/d}\\
= (|a_F|^d|b_F|^{d-d_0}\cdot|a_F|^{-d_1}|R(F_0(1,z),F_1(1,z))|)^{1/d}
=|\Res F|^{1/d}.
\end{gather*}
Hence on $\bC^2$,
\begin{gather}
 |F(p)\wedge (0,1)|=|\Res F|^{1/d}\prod_{q\in F^{-1}(0,1)}|p\wedge q|^{1/d}.\label{eq:factorization}
\end{gather}

From (\ref{eq:invariance}), $G^F(F(p))=d\cdot G^F(p)$ and
$G^F(q)=G^F(0,1)/d$ ($q\in F^{-1}(0,1)$).
Hence the $\log$ of (\ref{eq:factorization}) descends to $\bP^1$ as
\begin{gather*}
 \Phi_F(f(z),\infty)=\frac{1}{d}\log|\Res F|+\int_{\bP^1}\Phi_F(z,w)\rd(f^*\delta_{\infty})(w).\label{eq:factor}
\end{gather*}
Integrating this in $\rd\mu_f(z)$,
\begin{gather*}
 \int_{\bP^1}\Phi_F(f(z),\infty)\rd\mu_f(z)=\frac{1}{d}\log|\Res F|+\int_{\bP^1}U_{F,\mu_f}(w)\rd(f^*\delta_{\infty})(w),
\end{gather*}
and from $f_*\mu_f=\mu_f$ and $U_{F,\mu_f}\equiv V_F$,
\begin{align*}
 V_F&=\int_{\bP^1}\Phi_F(\infty,\cdot)\rd f_*\mu_f=
\int_{\bP^1}\Phi_F(\infty, f(z))\rd\mu_f(z)\\
&=\frac{1}{d}\log|\Res F|+\int_{\bP^1}V_F\rd(f^*\delta_{\infty})(w)
=\frac{1}{d}\log|\Res F|+d\cdot V_F.
\end{align*}
Now the proof of (\ref{eq:resultant}) is completed. \qed

\begin{remark}
 From (\ref{eq:factorization}),
 \begin{multline*}
  1=\prod_{p\in F^{-1}(1,0)}|F(p)\wedge (0,1)|
=|\Res F|^{(1/d)\cdot d^2}\prod_{p\in F^{-1}(1,0),q\in F^{-1}(0,1)}|p\wedge q|^{1/d},
 \end{multline*}
 and we also obtain an important formula
 \begin{gather*}
  |\Res F|=\prod_{p\in F^{-1}(1,0),q\in F^{-1}(0,1)}|p\wedge q|^{-1/d^2}.
 \end{gather*}
\end{remark}

\def\cprime{$'$}

\end{document}